\documentclass{article}

\usepackage{amsthm}
\usepackage{amsmath}
\usepackage{amsfonts}
\usepackage{hyperref}
\usepackage{cleveref}
  
\newtheorem{theorem}{Theorem}[section]
\newtheorem{lemma}[theorem]{Lemma}
\newtheorem{fact}[theorem]{Fact}
\newtheorem{notation}[theorem]{Notation}
\newtheorem{problem}[theorem]{Problem}

\usepackage{graphicx}  
\usepackage{authblk}

\newcommand{\GL}[0]{{\operatorname{GL}}}  
\newcommand{\gl}[0]{{\mathfrak{gl}}}  
\newcommand{\Aff}[0]{{\operatorname{Aff}}}  
\newcommand{\aff}[0]{{\mathfrak{aff}}}  
\newcommand{\PGL}[0]{{\operatorname{PGL}}}  
\newcommand{\pgl}[0]{{\mathfrak{pgl}}}
\renewcommand{\sl}[0]{{\mathfrak{sl}}}

\newcommand{\BZ}[0]{{\mathbb Z}}  
\newcommand{\BQ}[0]{{\mathbb Q}}  
\newcommand{\BR}[0]{{\mathbb R}}

\newcommand{\CC}[0]{{\mathcal C}}

\date{}

\title{Classification of maps sending lines into translates of a curve}

\author
{J\'{o}zsef Solymosi
\thanks{Department of Mathematics,
University of British Columbia, 1984 Mathematics Road,
Vancouver, BC, V6T 1Z2, Canada, and Obuda University,
Budapest, Hungary.
Email: solymosi@math.ubc.ca}
\and 
Endre Szab\'o
\thanks{  
Alfr\'ed R\'enyi Institute of Mathematics,
	 Hungarian Academy of Sciences
	 H-1053 Budapest, Re\'altanoda u. 13--15.
Email: endre@renyi.hu}}

\begin{document}

\maketitle

\begin{abstract}
We list four types of planar curves such that arrangements of their translates are (locally) combinatorially
equivalent to an arrangement of lines. We find them by characterising diffeomorphisms $\phi:\BR^2\to\BR^2$ and continuous
  curves $C\subset\BR^2$ such that
  $\phi\big( t+C \big)$ is a line for all $t\in\BR^2$.
There are exactly five maps satisfying (at least locally) this condition. Two of them define the same curve, so we have four different curves.
These can be used to define norms giving constructions with $\Omega(n^{4/3})$ unit distances among $n$ points in the plane. 
\end{abstract}



\section{Introduction}

\noindent 
One of the oldest and best-known problems in combinatorial
geometry is Paul Erd\H os' unit distances problem \cite{ER}.
What is the maximum number of unit distances among $n$
points on the plane? Erd\H os conjectured that the maximum
number of unit distances is $n^{1+o(1)}.$ 
(Through the paper we are going to use the Big-O, Little-o, and Omega notations. A real function $f(x)$
is $o(T(x))$ if $f(x)/T(x)\rightarrow 0$ as $x\rightarrow \infty.$ It is $O(T(x))$ if there is a $c>0$ such that $f(x)/T(x)\leq c$ as $x\rightarrow \infty,$ and it 
is $\Omega(T(x))$ if there is a $c>0$ such that $f(x)/T(x)\geq c$ as $x\rightarrow \infty.$)

The conjecture is
still open, the best-known upper bound is $O(n^{4/3}).$ This bound was
proved by Spencer, Szemer\'edi, and Trotter \cite{SST}. It
seems that the exponent ${4/3}$ is the limit of the known
combinatorial methods, even to prove $o(n^{4/3})$ is out
of range of the known techniques.

\medskip
One reason behind this barrier is that there are
norms where one can find $n$-element point-sets with
$\Omega(n^{4/3})$ unit distances. The oldest construction
providing such a norm can probably be derived from Jarn\'ik's construction \cite{Ja}.
Jarn\'ik defined a sequence of centrally symmetric smooth closed convex curves, $U_m$ 
containing $\Omega(m^{2/3})$ lattice points of the $m\times m$ integer grid. 
Setting such a convex curve as the unit disk, there are $\Omega(m^{8/3})$
unit distances among the $(2m)^2$ points of the $2m\times 2m$ integer lattice. 

\noindent
For completeness, we include the simple argument. Let's denote the centre of $U_m$ by $o.$ 
$U_m$ is centrally symmetric, so (by possibly losing a multiplier of 2) we can assume that $o\in [m]\times [m].$
The set, 
\[
\left\{(U_m+o)+(i,j) | i,j\in [m]\right\}
\]
has at most $(2m)^2$ points and every translate of $o$ has at least $\Omega(m^{2/3})$ points (the corresponding translate of $U_m$) 
at unit distance. If we set $n=m^2$, there are $\Omega(n^{4/3})$ unit distances with this norm.

In this example, the norm changes with $n.$ A nice construction, with a uniform norm, 
was given by Valtr \cite{VA} using translates of a parabola and the $n\times n^2$ integer grid. 
(see the description of the construction on page 194 in \cite{Pach}) On the other hand, it was 
proved by Matousek that most norms, in the sense of Baire category, determine $O(n\log n\log\log n)$ unit distances 
among $n$ points \cite{Ma}. This bound was improved recently to $O(n\log n)$ by Alon et al. in \cite{A}.

\medskip
For any strictly convex norm, among $n$ points in the plane, there are $O(n^{4/3})$ unit distances. 
This claim can be proved using the crossing number inequality from \cite{Szek} in the same way as proving the
Szemer\'edi-Trotter theorem. This theorem gives a sharp upper bound  on the number of incidences, $I,$
between $N$ points and $M$ lines on the real plane, $\mathbb{R}^2.$

\begin{theorem}[Szemer\'edi-Trotter Theorem \cite{SZT}]\label{SZTR}
$$I(N,M)=O\left(N^{\frac{2}{3}}M^{\frac{2}{3}}+N+M\right).$$
\end{theorem}

An elegant proof of the above theorem was given by Sz\'ekely, who also showed how to use his proof method
to give the $O(n^{4/3})$ bound for unit distances \cite{Szek}. 

\medskip
There are known
arrangements of $n$ lines and $n$ points such that
\begin{equation}
I(n,n)=\Omega\left(n^{\frac{4}{3}}\right).
\end{equation}
Such arrangements were found by Erd\H os, Elekes \cite{El}, Sheffer and Silier \cite{ShSi} using lattice points, i.e. a Cartesian product structure.
Recently Guth and Silier gave sharp examples not based on the integer lattice \cite{GS}.

\medskip
If there were maps where the images of lines are translates of a single curve, $\cal{C}$, then one can map such point-line arrangements 
to point-curve arrangements. Using part of the curve as (part of) the unit circle, with this norm we have  $\Omega(n^{4/3})$ unit distances.
Such a map exists: the map 
\begin{equation}\label{map1}
    (x,y)\mapsto (x,y+x^2)
\end{equation}
sends the line $(t,at+b)$ to the $(t,t^2+at+b)=\left(t,(t+a/2)^2-a^2/4+b\right)$ curve, which is 
a translate of the parabola $y=x^2$. This map was used over finite fields by Pudl\'ak in \cite{PP}. Pudl\'ak noticed that this map gives a 
one-to-one correspondence between point-parabola incidences and point-line incidences. He used it to define the colouring of the edges of a complete bipartite graph
with three colours without creating large monochromatic complete subgraphs.
If we apply the map in (\ref{map1}) to Elekes' point-line arrangement in \cite{El}
we get a construction very similar to Valtr's \cite{VA}.

\medskip
Based on the above observations, it is a natural problem characterizing maps of the plane sending lines into translates of a single curve. As we will see
there are exactly four more such maps in addition to Pudl\'ak's map in (\ref{map1}). 
The first and the third maps in the list below result in the same curve, these are translates of the log (exp) curve as the images of lines.
The last curve, which is not listed here, is given by the 
real and imaginary parts of the complex logarithm function.
\begin{enumerate}
    \item  $$M:(x,y)\mapsto(x,\ln(y)).$$
For every line $y=ax+b$ with $a>0,$ the $x>-\frac{b}{a}$ part maps to $\left(x,\ln\left(x+\frac{b}{a}\right)+\ln({a})\right).$ This is a translate of the curve $y=\ln(x).$    
\item  $$M:(x,y)\mapsto \left(\ln(x),\ln\left(\frac{y}{x}\right)\right).$$
    If we use the notation $\ln(x)=X,$ then the image of the line $y=ax+b$ is the $$\left(X,\ln\left(1+\frac{b/a}{e^X}\right)+\ln(a)\right)=\left(X,\ln\left(1+{e^{-(X-\ln(b/a))}}\right)+\ln(a)\right) $$
    curve if $a>0$ and $b>0.$ This is the translate of the curve $y=\ln(1+e^{-x})$.
    \item  $$M:(x,y)\mapsto \left(\ln(x),\frac{y}{x}\right).$$
    As in the previous case, if we use the notation $\ln(x)=X,$ then the image of the line $y=ax+b$ is the $$\left(X,e^{-(X-\ln(b))}+a\right) $$
    curve if $a>0$ and $b>0.$ This is the translate of the curve $y=e^{-x}$, so it is similar to the first, $y=\ln(x),$ case.
    \item This is probably the most surprising map 
    \[ M:(x,y)\mapsto\left(\Re(\ln(1-x+iy)),\Im(\ln(1-x+iy))\right).\]
    If a generic line is given by the equations $x=t, y= at+b$, then its image after the map is 
\[
x(t)=\Re\left(\ln\left(t-\frac{1+ib}{1-ia}\right)+\ln(ia-1)\right), \]
\[y(t)=\Im\left(\ln\left(t-\frac{1+ib}{1-ia}\right)+\ln(ia-1)\right),
\]
the real and imaginary parts of a translate of the complex logarithm function.

\end{enumerate}


\section{Preliminaries}
\label{sec:preliminaries}

\begin{notation}
  The group of invertible $3\times3$ matrices and the Lie algebra of
  all $3\times3$ matrices will be denoted by  $\GL(3,\BR)$ and $\gl(3,\BR)$
  resp. The quotient group of $\GL(3,\BR)$ by the normal subgroup of
  scalar matrices is denoted by $\PGL(3,\BR)$,
  it is the group of projective linear transformations of the
  projective plane.
  The Lie algebra of $\PGL(3,\BR)$ is the quotient of $\gl(3,\BR)$ by
  the ideal of scalar matrices, we denote it with $\pgl(3,\BR)$.
  It is naturally isomorphic to $\sl(3,\BR)$,
  the Lie algebra of $3\times3$ matrices of trace $0$.
\end{notation}

\begin{notation} \label{notation:affine-group}
  We denote by $\Aff(2,\BR)\le\GL(3,\BR)$ the subgroup of
  all $3\times3$ matrices of the form
  $$
  \left(\begin{matrix}
      L&v\\
      0&1\\
    \end{matrix}\right)
  $$
  where $L$ is a $2\times2$ invertible matrix (linear transformation),
  and $v$ is a 2-dimensional column vector (translation).
  We denote by $\aff(2,\BR)\le\gl(3,\BR)$ the Lie subalgebra of
  matrices of the form
  $$
  \left(\begin{matrix}
      \Lambda&v\\
      0&0\\
    \end{matrix}\right)
  $$
  where $\Lambda$ is a $2\times2$ matrix,
  and $v$ is a 2-dimensional column vector.
  The quotient homomorphism $\GL(3,\BR)\to\PGL(3,\BR)$ maps
  $\Aff(2,\BR)$ isomorphically onto its image, and similarly,
  the quotient homomorphism $\gl(3,\BR)\to\pgl(3,\BR)$ maps
  $\aff(2,\BR)$ isomorphically onto its image. We shall often
  identify $\Aff(2,\BR)$ and $\aff(2,\BR)$ with these images.
  With this identification $\Aff(2,\BR)$ becomes the group of affine
  transformations of $\BR^2$, and $\aff(2,\BR)$ becomes its Lie
  algebra.
\end{notation}

The following is well-known.
\begin{fact} \label{line-preserving-maps}
  Let $\psi_0:\BR^2\to\BR^2$ be a diffeomorphism
  which maps all lines into lines.
  Then $\psi_0\in\Aff(2,\BR)$.
\end{fact}
We need the following local version of this.
\begin{lemma} \label{locally-line-preserving-maps}
  Let $\psi:U\to V$ be a diffeomorphism
  between connected open subsets $U,V\subseteq\BR^2$
  which maps all line segments in $U$ into line segments in $V$.
  Then $\psi$ can be uniquely extended into a projective linear map
  $\tilde\psi\in\PGL(3,\BR)$.
\end{lemma}

\begin{proof}
  The standard proof of \Cref{line-preserving-maps} works here if one is
  careful enough. We recall it for the sake of completeness.
  
  Let $S\subset U$ be a line segment.
  We define the following set of real numbers.
  \begin{equation*}
    \label{eq:1}
    \CC_S=\left\{ \lambda\in\BR\;\Bigg|\;
      {\text{If }A,B,C,D\in S\text{ \ with cross-ratio }
        (A,B:C,D)=\lambda
        \atop
        \text{then }\big(\psi(A),\psi(B);\psi(C),\psi(D)\big)=\lambda.
      }\right\}
  \end{equation*}
  We make several observations.
  \begin{enumerate}
  \item \label{item:1}
    If $\lambda\in\CC_S$ then $1-\lambda\in\CC_S$.\\
    Indeed, $(A,B:C,D) = 1-(A,C:B,D)$.
  \item \label{item:2}
    If $\lambda<0<\mu$ and $\lambda,\mu\in\CC_S$ then
    $\frac\mu\lambda\in\CC_S$.\\
    Indeed, let $(A,B:C,D)=\frac\mu\lambda$.
    Since this is negative, one of $C$ and $D$ lies inside
    $\overline{AB}$, the other lies outside.
    If $C\in\overline{AB}$ then we relabel $A,B,C,D$ to $B,A,D,C$,
    this does not change their  cross-ratio.
    So we can assume that $C\notin\overline{AB}$.
    Then there is a unique $E\in\overline{AB}$ with
    $(A,B:E,C))=\lambda$, and an easy calculation shows that
    $(A,B:E,D)=\mu$.
    This implies that
    $\big(\psi(A),\psi(B);\psi(E),\psi(C)\big)=\lambda$
    and $\big(\psi(A),\psi(B);\psi(E),\psi(D)\big)=\mu$,
    hence
    $\big(\psi(A),\psi(B);\psi(C),\psi(D)\big)=\frac\mu\lambda$.
  \item \label{item:3}
    $0,1\in\CC_S$.\\
    Indeed, if $(A,B:C,D)=0$ then either $A=C$ or $D=B$.
    This implies that either $\psi(A)=\psi(C)$ or $\psi(D)=\psi(B)$,
    hence $\big(\psi(A),\psi(B);\psi(C),\psi(D)\big)=0$.
    Therefore $0\in\CC_S$, and then (\ref{item:1}) implies $1\in\CC_S$.
  \item \label{item:4}
    $-1\in\CC_S$. \\
    Indeed, let $W\subseteq U$ be a convex neighborhood of
    $S$. If $(A,B:C,D)=-1$ then there is a complete quadrangle in $W$ which
    justifies this, i.e. $A$ and $C$ are the intersection points of
    the opposite sides, and the diagonals intersect the $AB$ line at
    $C$ and $D$. Then $\psi$ maps this quadrangle to a quadrangle in
    $V$ justifying that
    $\big(\psi(A),\psi(B);\psi(C),\psi(D)\big)=\lambda$.
  \item \label{item:5}
    $\BZ\subseteq\CC_S$.\\
    Indeed, starting with $-1\in\CC_S$, and applying (\ref{item:1})
    and (\ref{item:2}) alternately, we obtain that
    $2,-2,3,-3,4,-4,\dots\in\CC_S$. 
  \item
    $\BQ\subseteq\CC_S$.\\
    Indeed, negative rational numbers are in $\CC_S$
    by (\ref{item:5}) and (\ref{item:2}).
    Then by (\ref{item:1}), the nonnegative rational numbers also
    belong to $\CC_S$.
  \item
    $\CC_S=\BR$.\\
    Indeed, by the continuity of the cross-ratio, $\CC_S$ is a closed set.
  \end{enumerate}
  Now let $W\subseteq U$ be a convex open subset.
  The above observations imply that $\psi$ preserves all cross-ratios
  in $W$, hence there is a unique projective linear map
  $\overline\psi_W\in\PGL(3,\BR)$ which agrees with $\psi$ on $W$. 
  For overlapping convex open sets, the corresponding projective linear
  maps must be equal. Since $U$ is connected, all these
  $\overline\psi_W$ must be equal.
  This proves the lemma.
\end{proof}

\begin{lemma} \label{commuting-matrices-aonjugated-into-aff}
  Let $A,B\in\gl(3,\BR)$ be matrices whose images
  $\overline A,\overline B\in\pgl(3,\BR)$ commute.
  Then, after a suitable base change,
  $\overline A,\overline B\in\aff(2,\BR)$.
\end{lemma}
\begin{proof}
  Commutators have trace zero.
  So $[A,B]$ is a scalar matrix with trace 0,
  hence $A$ and $B$ commute in $\gl(3,\BR)$.
  We distinguish three cases.
  \begin{itemize}
  \item
    If $A$ has one real and two conjugate complex eigenvalues,
    then let $V$ be the real part of the linear span of the complex
    eigenspaces corresponding to the non-real eigenvalues of $A$.
  \item
    If $A$ has three different real eigenvalues, then let $V$ be the
    linear span of any two of the corresponding eigenspaces.
  \item
    If $A$ has only two different real eigenvalues, then let $V$ be the
    eigenspace corresponding to the eigenvalue with multiplicity two.
  \item
    Otherwise, $A$ has a single real eigenvalue of multiplicity three,
    hence $A$ is a scalar matrix. In this case, we switch the role of $A$
    and $B$, and go through this list again.
    If $B$ is not a scalar matrix then we obtain our $V$.
  \item
    Finally, if both $A$ and $B$ are scalar matrices, then let $V$ be an
    arbitrary plane in $\BR^3$.
  \end{itemize}
  In all cases, $V$ is a plane in $\BR^3$ invariant under both $A$ and
  $B$. After an appropriate base change, $V$ will be the hyperplane
  of vectors whose last coordinate is zero.
  Matrices that map this $V$ into itself
  are of the form
  $$
  \left(\begin{matrix}
      *&*&*\\
      *&*&*\\
      0&0&*
    \end{matrix}\right)
  $$
  where elements marked with $*$ are arbitrary,
  and the quotient homomorphism $\gl(3,\BR)\to\pgl(3,\BR)$
  maps such matrices into $\aff(2,\BR)$.
  Hence, after this base change,
  $\overline A,\overline B\in\aff(2,\BR)$.
\end{proof}

\section{Affine Structures}

\begin{notation}
  For a plane curve $C\subset\BR^2$ and a vector $t\in\BR^2$
  we denote by $t+C$
  the translate of $C$ with $t$.
\end{notation}

\begin{problem} \label{problem:phi0}
  Characterise diffeomorphisms $\phi_0:\BR^2\to\BR^2$ and continuous
  curves $C\subset\BR^2$ such that
  $\phi_0\big( t+C \big)$ is a line for all $t\in\BR^2$.
\end{problem}

\begin{problem} \label{problem:phi}
  Characterize continuous curves $C\subset\BR^2$ and
  diffeomorphisms $\phi:U\to V$
  between connected, open subsets $U,V\subseteq\BR^2$
  such that
  $\phi\big( (t+C) \cap U\big)$ is contained in a line for all $t\in\BR^2$.
\end{problem}

By composing $\phi_0$ on the left with an appropriate translation,
one can reduce Problem~\ref{problem:phi0} to the case when
$\phi_0(0,0)=(0,0)$.
Similarly,
by composing $\phi$ on both sides with appropriate translations
one can reduce Problem~\ref{problem:phi} to the case when
$(0,0)\in U$ and $\phi(0,0)=(0,0).$

Let $T\cong\BR^2$ denote the group of translations of $\BR^2$.
If we conjugate $T$ with any of these $\phi_0$
then by \Cref{line-preserving-maps}
we arrive at a subgroup of $\Aff(2,\BR)$.
Since $T$ is connected, these subgroups are uniquely determined by
their Lie algebras,
which are 2-dimensional commutative subalgebras of $\aff(2,\BR)$.

In the more general setup,
if we conjugate a small neighbourhood of the identity in $T$ with
any of these $\phi$
then by \Cref{locally-line-preserving-maps}
we arrive at a small neighbourhood of the identity in a connected subgroup
of $\PGL(2,\BR)$.
Again, these subgroups are uniquely determined by their Lie algebras,
which are, in this case, 2-dimensional commutative subalgebras of
$\pgl(2,\BR)$.
By \Cref{commuting-matrices-aonjugated-into-aff} these subalgebras,
after a base change, become subalgebras of $\aff(2,\BR)$.

So in both problems, we need to classify 2-dimensional
commutative subalgebras of $\aff(2,\BR)$, and analyze whether the
corresponding connected subgroups are isomorphic to $T$ or not.

The embedding of a Lie algebra $\mathfrak{g}$ into $\aff(2,\BR)$
is also called an \emph{affine structure on $\mathfrak{g}$}.
Such an embedding gives rise to an affine action on $\BR^2$
of the simply connected Lie group with Lie algebra $\mathfrak{g}$,
which we call the \emph{corresponding action}.
In particular, a 2-dimensional abelian subalgebra of $\aff(2,\BR)$
is called an affine structure on the 2-dimensional abelian Lie
algebra, and the corresponding action is an affine action of $T$
on $\BR^2$.

Affine structures on the 2-dimensional abelian Lie algebra were analyzed in the work of Rem and Goze \cite{RG}
where they proved that there are six affinely non-equivalent affine structures. They listed the affine structures on the 2-dimensional Lie
algebra and the corresponding action. Based on their list, there are
six actions we have to check for a potential $\phi$ or  $\phi_0$.

\subsection{The six affine actions}

In what follows, we are checking the affine actions listed in \cite{RG}
whether they give rise to a solution of Problem~\ref{problem:phi0} or
Problem~\ref{problem:phi}.
By definition, these affine actions are homomorphisms from
the group of translations $T\cong\BR^2$ into $\Aff(2,\BR)$.
Let $A(s,t)\in\Aff(2,\BR)$ denote the homomorphic image
of the element $(s,t)\in T$.
Recall from Notation~\ref{notation:affine-group}
that $A(s,t)$ is a $3\times3$ matrix of certain special form.

For every affine action,
we are looking for a diffeomorphism $\phi:U\to V$
between connected open neighbourhoods $U,V$ of the origin in $\BR^2$
such that $\phi(0,0)=(0,0)$ and
\begin{equation}\label{action}
\phi^{-1}(\phi(x,y)+(s,t))=A^*(s,t)\cdot[x,y,1]^T,
\end{equation}
where $A^*(s,t)$ is the $2\times 3$ submatrix of $A(s,t)$ containing
the first two rows.
We will not specify $U$ and $V$, only their existence is important to us.
However, by studying  the formulas we have for $\phi$ the reader can easily
find appropriate $U$ and $V$.

\medskip
Let us denote the translation by $(s,t)$ as $T(s,t) : \BR^2 \rightarrow \BR^2$    and let \\ $a(s,t) : \BR^2 \rightarrow \BR^2$  denote the affine transformation as defined above with the $A(s,t)$ matrix.   With this notation equation (\ref{action}) becomes 
\[
\phi^{-1}  \circ  T(s,t)  \circ  \phi  =  a(s,t)
\] 
where $\circ$ denotes the function composition. After rearranging it, we have 
\[
\phi^{-1}  \circ  T(s,t)   =  a(s,t) \circ  \phi^{-1}.
\] 
Applying both sides to $(x,y)=(0,0)$  we obtain
\[
  \phi^{-1}(s,t) = a(s,t)(0,0)=A^*(s,t)\cdot[0,0,1]^T,
\]
so we can recover $\phi$ from
the first two entries in the last column of $A(s,t)$.

In the first two cases our $\phi$ is actually an $\BR^2\to\BR^2$
diffeomorphism,
in the last four cases, we get local maps giving solutions in the selected range.

\begin{enumerate}
    \item Identity (case $A_5$ in \cite{RG})
   \[A(s,t)=\begin{pmatrix}
1 & 0 & s \\
0 & 1 & t \\
0 & 0 & 1 \\
\end{pmatrix}\]
Here $\phi$ is the identity map, it won't define the translates of a single curve.

    \item Parabola (case $A_4$ in \cite{RG})
   \[A(s,t)=\begin{pmatrix}
1 & 0 & s \\
s & 1 & t+\frac{s^2}{2} \\
0 & 0 & 1 \\
\end{pmatrix}\]
Here $\phi:(x,y)\mapsto\left(x,y+\frac{x^2}{2}\right).$ Every line $(x,ax+b)$ maps to a translate of a parabola $y=\frac{x^2}{2}.$ Its image is $\left(x,\frac{(x+a)^2}{2}+\frac{a^2}{2}+b\right).$ This is Pudl\'ak's map in \cite{PP} we mentioned in the introduction.

\medskip

    \item Log curve (case $A_6$ in \cite{RG})
   \[A(s,t)=\begin{pmatrix}
e^s & 0 & e^s-1 \\
0 & 1 & t \\
0 & 0 & 1 \\
\end{pmatrix}\]
Here $\phi:(x,y)\mapsto(\ln(x+1),y).$ 
It is more convenient to work with the map $\phi':(x,y)\mapsto(x,\ln(y)),$ giving the same family of curves.
For every line $y=ax+b$ with $a>0,$ the part $x>\frac{b}{a}$ maps to a translate of a log curve, $y=\ln(x).$ Its image is $\left(x,\ln\left(x+\frac{b}{a}\right)+\ln({a})\right).$

    \item (case $A_1$ in \cite{RG})
   \[A(s,t)=\begin{pmatrix}
e^s & 0 & e^s-1 \\
e^s(e^t-1) & e^se^t & e^s(e^t-1) \\
0 & 0 & 1 \\
\end{pmatrix}\]
Here $\phi:(x,y)\mapsto\left(\ln(x+1),\ln\left(1+\frac{y}{x+1}\right)\right).$ 

    The image of the part $x>\max\left(-\frac{b+1}{a+1},-1\right)$ of the $(x,ax+b)$ line is a curve. However, different lines map into different types of curves. Lines of the form $y=c(x+1)$ map to a horizontal line $y=\ln(1+c).$ This is still an interesting map. Let us consider a line $y=ax+b$ where $a\neq -1$ and $\frac{b-a}{1+a}>0.$ If we set $\ln(x+1)=X,$ then the image curve is 
    
    \[
    \left(X,\ln\left(1+a+(b-a)e^{-X}\right)\right)= \left(X,\ln\left(1+e^{-\left(X-\ln\left(\frac{b-a}{1+a}\right)\right)}\right)+\ln(1+a)\right),
    \]
which is a translate of the curve $y=\ln(1+e^{-x})$.

\item (case $A_2$ in \cite{RG})
   \[A(s,t)=\begin{pmatrix}
e^s & 0 & e^s-1 \\
e^st & e^s & e^st \\
0 & 0 & 1 \\
\end{pmatrix}\]
Here $\phi:(x,y)\mapsto\left(\ln(x+1),\frac{y}{x+1}\right).$ 

   As in the previous case, lines of the form $y=c(x+1)$ map to a horizontal line. But if we consider lines $ax+b,$ where $b>a\neq 0,$ and set $\ln(x+1)=X,$ then the image curve is 
    
    \[
    \left(X,a+\frac{b-a}{e^X}\right)= \left(X, e^{-(X-\ln(b-a))}+a\right),
    \]
which is a translate of the curve $y=e^{-x}$. 

\item Rotations (case $A_3$ in \cite{RG})
   \[A(s,t)=\begin{pmatrix}
e^s\cos t & -e^s\sin t & 1-e^s\cos t \\
e^s\sin t & e^s\cos t & e^s\sin t \\
0 & 0 & 1 \\
\end{pmatrix}\]
Images of translations with $(s,t+2k\pi)$ would give the same line for any integer $k.$ Also, lines 
of the form $y=c(x+1)$ map to a horizontal line. 

In this last case, the map is given by 
\[\phi:(x,y)\mapsto\left(\frac{\ln((1-x)^2+y^2)}{2},\arctan\left(\frac{y}{1-x}\right)\right),\]

or, equivalently by

\[
\phi:(x,y)\mapsto\left(\Re(\ln(1-x+iy)),\Im(\ln(1-x+iy))\right).
\]

If a line is given by the $x=t, y= at+b$ equations

then its image after the map is 
\[
x(t)=\Re\left(\ln\left(t-\frac{1+ib}{1-ia}\right)+\ln(ia-1)\right), \]
\[y(t)=\Im\left(\ln\left(t-\frac{1+ib}{1-ia}\right)+\ln(ia-1)\right),
\]

which is the real and the imaginary part of a translate of the complex logarithm function.

\begin{figure}[h]\label{last}
  \begin{center}
   \includegraphics[scale=0.3]{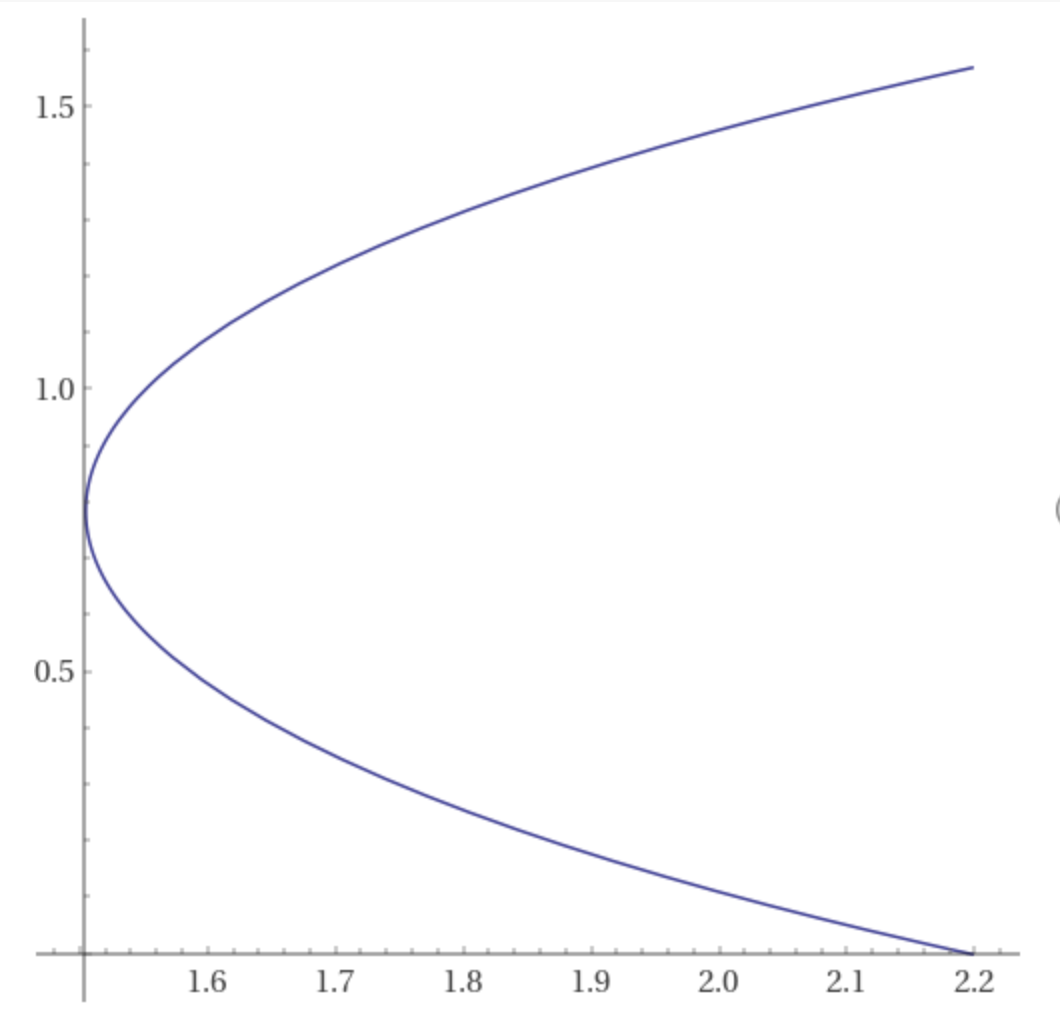}
  \caption{Part of the curve from Case 6}
  \end{center}
\end{figure}

From the six affine actions, we found four different curves such that arrangements of their translates are (locally) combinatorically equivalent to arrangements of lines. 
\end{enumerate}

\section{Concluding remarks and open problems}

\begin{itemize}
    
    \item One can ask the same problem in higher dimensions. What are the surfaces such that any finite arrangement of hyperplanes is combinatorically equivalent to translates of the surface?
    Remm and Goze classified the three-dimensional commutative, associative real algebras in \cite{RG}, so based on their work, one can characterize such surfaces in three-space. One of these (an extension of Pudl\'ak's map \cite{PP}) was used by Zahl in \cite{Z} to construct a norm determining $\Omega(n^{3/2})$ unit distances among $n$ points. In the same paper, Zahl was able to break the $n^{3/2}$ barrier, showing that for the Euclidean norm, the number of unit distances determined by $n$ points is $O(n^{3/2-c})$ for some $c>0.$ In dimension four and higher there are $n$-element pointsets with $\Omega(n^2)$ unit distances (see more about such constructions in \cite{Konrad}).
    
    \item It follows from Sz\'ekely's proof of the Szemer\'edi-Trotter theorem in \cite{Szek} that $m$ translates of  a convex curve and $n$ points determine $O(n^{2/3}m^{2/3}+n+m)$ incidences (see also in \cite{ENR}). In this paper, we listed four curves where the above incidence bound is sharp, i.e. for each curve, there are arrangements of $m$ translates of the curve and $n$ points  with $\Omega(n^{2/3}m^{2/3}+n+m)$ incidences for arbitrarily large $n$ and $m$. Are there such planar curves significantly different from the ones listed above? 
    
\end{itemize}

\medskip

\section{Acknowledgements}
JS's research was partly supported by a Hungarian National Research Grant KKP grant no.133819, an NSERC Discovery grant and OTKA K  grant no.119528.
ESz was supported by the National Research, Development and Innovation
Office (NKFIH) Grant K138596. The project leading to this application has received funding from the European Research Council (ERC) under the European Union's Horizon 2020 research and innovation programme (grant agreement No 741420)

The comments and suggestions of the referee greatly improved the presentation of the results. We are thankful for that.


\begin{thebibliography}{10}

\bibitem{A} N. Alon, M. Buci\'c, and L. Sauermann, Unit and distinct distances in typical norms (2023) 
arXiv:2302.09058 [math.CO]

\bibitem{Pach} P. Brass, W.O.J. Moser, and J. Pach, Research Problems in Discrete Geometry,
Springer (2005) 1st. edition, XII, 500 pp.

\bibitem{ENR} Gy. Elekes, M. B Nathanson, and I. Z Ruzsa,
Convexity and Sumsets,
Journal of Number Theory,
Volume 83, Issue 2,
(2000) 194--201,

\bibitem{El} Gy. Elekes, On linear combinatorics I. concurrency -- an algebraic approach. 
Combinatorica, (2011) 17(4):447--458, 

\bibitem{ER} P. Erd\H os, On sets of distances of n points.
American Mathematical Monthly (1946) 53: 248--250.

\bibitem{Ja}
V.V. Jarn\'ik, \"Uber die Gitterpunkte auf konvexen Kurven. Mathematische Zeitschrift 24 (1926): 500--518. 

\bibitem{Konrad}
K.J. Swanepoel, Unit Distances and Diameters in Euclidean Spaces. Discrete Comput Geom (2009) 41, 1--27.

\bibitem{Ma} J. Matou\v{s}ek,  
The number of unit distances is almost linear for most norms. Advances in Mathematics. (2011) 226. 2618--2628. 

\bibitem{PA} J. Pach and P.K. Agarwal, Combinatorial geometry.
Wiley-Interscience Series in Discrete Mathematics and Optimization.
A Wiley-Interscience Publication. John Wiley \& Sons, Inc., New York,
1995. xiv+354 pp. ISBN: 0-471-58890-3

\bibitem{PP} P. Pudl\'ak, On explicit Ramsey graphs and estimates of the number
of sums and products. in: Topics in Discrete Mathematics,
eds. Klazar, Kratochvil, Loebl, Matousek, Thomas and Valtr.
Springer 2006, 169--175.

\bibitem{RG} E. Remm and M. Goze,
Affine structures on abelian Lie groups,
Linear Algebra and its Applications,
Volume 360,
(2003) 215--230.

\bibitem{GS} L. Guth and O. Silier, Sharp Szemer\'edi-Trotter Constructions in the Plane,
arXiv:2112.00306 [math.CO]

\bibitem{SZT} E. Szemer\'edi and W. Trotter, Extremal problems in discrete geometry.
Combinatorica 3 (1983) 381--392.

\bibitem{SST} J. Spencer, E. Szemer\'edi, and W. Trotter,
Unit distances in the Euclidean plane.  Graph theory and combinatorics
(Cambridge, 1983),   Academic Press, London, (1984) 293--303.

\bibitem{ShSi} A. Sheffer and O. Silier, A structural Szemer\'edi-Trotter Theorem for Cartesian Products,
	arXiv:2110.09692 [math.CO]
	
\bibitem{Szek} L. Sz\'ekely, Crossing Numbers and Hard Erd\H os Problems in Discrete Geometry. Combinatorics, Probability and Computing, (1997) 6(3), 353--358.

\bibitem{VA} P.~Valtr, Strictly convex norms allowing many unit distances
and related touching questions, manuscript.

\bibitem{Z} J. Zahl, Breaking the 3/2 barrier for unit distances in three dimensions,
International Mathematics Research Notices, Volume 2019, Issue 20, (2019)  6235--6284.

\end{thebibliography}
\end{document}